\documentclass[letter,reqno,oneside,11pt]{amsart}
\usepackage{fullpage}
\usepackage{setspace}
\usepackage{datetime}

\usepackage{amsmath}
\usepackage{amsfonts}
\usepackage{amssymb}
\usepackage{verbatim,enumitem,graphics}
\usepackage{microtype}
\usepackage{xcolor}
\usepackage[backref=page]{hyperref}
\hypersetup{%
    colorlinks,
    linkcolor={red!50!black},
    citecolor={green!50!black},
    urlcolor={blue!50!black}
}
\usepackage{dsfont}
\usepackage{url}

\usepackage[utf8]{inputenc}

\usepackage[abbrev,msc-links,backrefs]{amsrefs} 
\usepackage{doi}

\renewcommand{\PrintDOI}[1]{\doi{#1}}

\usepackage{enumitem}
\usepackage{MnSymbol}

\setenumerate{leftmargin=*} 

\newtheorem{theorem}{Theorem}[section]
\newtheorem{lemma}[theorem]{Lemma}
\newtheorem{proposition}[theorem]{Proposition}
\newtheorem{fact}[theorem]{Fact}

\newtheorem{corollary}[theorem]{Corollary}
\newtheorem{conjecture}[theorem]{Conjecture}

\newtheorem{remark}[theorem]{Remark}

\newcommand{\oldqed}{}
\def\endofClaim{\hfill\scalebox{.6}{$\Box$}}

\newcommand{\NN}{\mathbb{N}}

\newcommand{\cD}{\mathcal{D}}

\newcommand{\cF}{\mathcal{F}}

\newcommand{\cP}{\mathcal{P}}

\newcommand{\cS}{\mathcal{S}}

\newcommand{\cQ}{\mathcal{Q}}

\newcommand{\eps}{\varepsilon}
\let\epsilon\varepsilon

\newcommand{\cC}{\mathcal{C}}

\newcommand{\Clm}{C_\ell(1,\ldots,1,K)}

\renewcommand{\subset}{\subseteq}
\newcommand{\dcup}{\dot\cup}

\usepackage{lineno}
\newcommand*\patchAmsMathEnvironmentForLineno[1]{%
\expandafter\let\csname old#1\expandafter\endcsname\csname #1\endcsname
\expandafter\let\csname oldend#1\expandafter\endcsname\csname end#1\endcsname
\renewenvironment{#1}%
{\linenomath\csname old#1\endcsname}%
{\csname oldend#1\endcsname\endlinenomath}}%
\newcommand*\patchBothAmsMathEnvironmentsForLineno[1]{%
\patchAmsMathEnvironmentForLineno{#1}%
\patchAmsMathEnvironmentForLineno{#1*}}%
\AtBeginDocument{%
\patchBothAmsMathEnvironmentsForLineno{equation}%
\patchBothAmsMathEnvironmentsForLineno{align}%
\patchBothAmsMathEnvironmentsForLineno{flalign}%
\patchBothAmsMathEnvironmentsForLineno{alignat}%
\patchBothAmsMathEnvironmentsForLineno{gather}%
\patchBothAmsMathEnvironmentsForLineno{multline}%
}

\begin{document}
\onehalfspace
\shortdate
\yyyymmdddate
\settimeformat{ampmtime}
\footskip=28pt


\title{Finding any given 2-factor in sparse pseudorandom graphs efficiently}

\author[J. Han]{Jie Han}
\author[Y. Kohayakawa]{Yoshiharu Kohayakawa}
\author[P. Morris]{Patrick Morris}
\author[Y. Person]{Yury Person}

\thanks{%
  JH was supported by FAPESP (2014/18641-5, 2013/03447-6).
  YK was partially supported by FAPESP (2013/03447-6) and
    CNPq (310974/2013-5, 311412/2018-1, 423833/2018-9).
  PM is supported by a Leverhulme Trust Study Abroad
  Studentship (SAS-2017-052$\backslash$9).  YP is supported by the
  Carl Zeiss Foundation.  The cooperation of the authors was supported
  by a joint CAPES-DAAD PROBRAL project (Proj.\ no.~430/15, 57350402,
  57391197).
  FAPESP is the S\~ao Paulo Research Foundation.  CNPq is the National
  Council for Scientific and Technological Development of Brazil.%
}

\address{Department of Mathematics, University of Rhode Island, 5 Lippitt Road, Kingston, RI, USA, 02881}
\email{jie\_han@uri.edu}

\address{Instituto de Matem\'atica e Estat\'istica, Universidade de
  S\~ao Paulo, Rua do Mat\~ao 1010, 05508-090 S\~ao Paulo, Brazil}
\email{yoshi@ime.usp.br}

\address{Institut f\"ur Mathematik, Freie Universit\"at Berlin, Arnimallee 3, 14195 Berlin, Germany and Berlin Mathematical School, Germany}
\email{pm0041@mi.fu-berlin.de}

\address{Institut f\"ur Mathematik, Technische Universit\"at Ilmenau, 98684 Ilmenau, Germany}
\email{yury.person@tu-ilmenau.de}

\begin{abstract}
Given an $n$-vertex pseudorandom graph $G$ and an $n$-vertex graph $H$ with maximum
degree at most two, we wish to find a copy of $H$ in $G$, i.e.\ an embedding $\varphi\colon V(H)\to V(G)$ so that $\varphi(u)\varphi(v)\in E(G)$ for all $uv\in E(H)$.  Particular
instances of this problem include finding a triangle-factor and
finding a Hamilton cycle in $G$. 
{Here, we provide a deterministic polynomial time algorithm that finds a given $H$ in any suitably pseudorandom graph $G$.}
The pseudorandom graphs we consider are $(p,\lambda)$-bijumbled graphs of minimum degree which is a constant proportion of the average degree, i.e.\ $\Omega(pn)$. A $(p,\lambda)$-bijumbled graph is characterised through the discrepancy property: $\left|e(A,B)-p|A||B|\right |<\lambda\sqrt{|A||B|}$ for any two sets of vertices $A$ and $B$. Our condition $\lambda=O(p^2n/\log n)$ on bijumbledness is within a log factor from being tight and provides a positive answer to a recent question of Nenadov. 

We combine novel variants of the absorption-reservoir method, a powerful tool from extremal graph theory and random graphs. Our approach builds on our previous work (\emph{European Journal of Combinatorics} \textbf{82} (2019), 102999), incorporating the work of Nenadov (\emph{Bulletin of the London Mathematical Society} \textbf{51}  (3) (2019), pp.~421--430), together with additional ideas and simplifications.
\end{abstract}

\maketitle


\section{Introduction}

A pseudorandom graph of edge density $p$ is a deterministic graph 
which shares typical properties of the corresponding random graph $G(n,p)$. 
These objects have attracted considerable attention in computer science and mathematics. 
 Thomason~\cite{Th87a,Th87b} was the first to introduce a quantitative notion of a pseudorandom graph by defining so-called 
$(p,\lambda)$-jumbled graphs $G$ which satisfy $\left| e(U)-p\binom{|U|}{2}\right|\leq \lambda|U|$ {for every vertex subset $U\subseteq V(G)$}. 
Ever since{, there has been a great deal of} investigation {into the } properties of pseudorandom graphs  
 {and this is still a very active area of modern research}. 

The most widely studied class of jumbled graphs are the so-called $(n,d,\lambda)$-graphs, which were introduced by Alon in the 80s. These graphs have $n$ vertices, are $d$-regular and their 
second largest eigenvalue in absolute value is at most~$\lambda$. An $(n,d,\lambda)$-graph satisfies the \emph{expander mixing lemma}~\cite{AC88} allowing good control of the edges between any two sets {of vertices} $A$ and $B$:
  \begin{equation}
    \left|e(A,B)-\frac{d}{n}|A||B|\right|<\lambda\sqrt{|A||B|},\label{eq:EML}
  \end{equation}
  where 
$e(A,B)=e_G(A,B)$ denotes the number of pairs\footnote{{Note that edges
in $A\cap B$ are counted twice.}}
$(a,b)\in A\times B$ so that $ab$ is an edge of~$G$. 
An illuminating survey of Krivelevich and Sudakov~\cite{KS06} provides  a wealth  of applications.

There are three interesting regimes in the study of pseudorandom graphs and the class of $(n,d,\lambda)$-graphs is versatile enough
to capture the {essence} of all of these regimes. 
In the first{,} one assumes $\lambda=\eps n$, where $n$
is the number of vertices in a graph $G$ and $\eps>0$ is an arbitrary fixed parameter. In this regime one can control edges between sets of linear sizes. 
This is tightly connected to the theory of quasirandom graphs~\cite{CGW89} and the applications of the regularity lemma of Szemer\'edi~\cite{KSSS02}. The second regime is when $d$ is constant and $\lambda<d${. T}his class then contains (non-bipartite) expanders~\cite{HLW06} and 
 Ramanujan graphs~\cite{LPS88}, which are 
 prominent objects of study {throughout mathematics and computer science}. 
The third regime (sparse graphs) concerns  $\lambda$ being $o(n)$, often some power of $n$, where 
 one has  better control on the distribution of edges between truly smaller sets.  
 This case has been investigated more recently and made amenable to some tools from extremal combinatorics \cite{ABHKP16, ABHKP17,CFZ12,HKP18a,HKP18b, HKMP18,KRSSS07,KS06, KriSudHam,KSS04,Nen18}.

The focus of this paper will be on conditions under which certain spanning or almost
  spanning structures are forced in sparse pseudorandom graphs. Our main motivation comes from probabilistic and extremal combinatorics, in particular the problem of universality.  A graph $G$ is called $\cF$-universal for some family $\cF$ if any member $F\in \cF$ can be embedded into $G$. This problem attracted a lot of attention~\cite{ACKRRSz00, AlCa07,AlCa08,Alon10}, especially {for the case where $\cF$ is a class} of bounded degree spanning subgraphs. {In this case we say an $n$-vertex graph $G$ is $\Delta$-universal if it contains all graphs on at most $n$ vertices of maximum degree $\Delta$.} {A large part of the focus of the study has been on the} 
  universality properties of $G(n,p)$ 
  ~\cite{ACKRRSz00, DKRR14, KL14, CFNS16,FKL16,FN17}. It is also natural to investigate the universality properties of $(n,d,\lambda)$-graphs as {was }suggested by Krivelevich, Sudakov and Szab\'o in~\cite{KSS04}.
  {In this setting of sparse pseudorandom graphs, }a general result on 
  universality has been proved only recently in~\cite{ABHKP16}. Let us comment that the case of dense graphs is 
{well}  understood since  the blow-up lemma of  Koml\'os,
S\'ark\"ozy and Szemer\'edi~\cite{KSS_bl} establishes that pseudorandom graphs of linear minimum degree contain any given bounded degree spanning structure.  A little later, the second and fourth author established jointly with Allen, B\"ottcher and  H\`an in~\cite{ABHKP16}{,} a variant of a  blow-up lemma for regular subgraphs of pseudorandom graphs.
{This }provides a machinery, complementing the results of Conlon, Fox and Zhao~\cite{CFZ12} and allowing to transfer many results about dense {graphs} to sparse graphs in a unified way. {However, these results are very general and thus do not establish tight conditions for special cases of spanning structures.}

{Much more is known for questions about finding one particular spanning structure in a pseudorandom graph and the} most prominent  
spanning structures which were considered in the last fifteen years include perfect matchings, studied by Alon, Krivelevich and Sudakov
in~\cite{KS06}, Hamilton cycles studied by Krivelevich and Sudakov~\cite{KriSudHam}, clique-factors~\cite{KSS04, HKP18b, HKMP18, Nen18} and powers of Hamilton cycles~\cite{ABHKP17}.

  The 
  {problem} of when a triangle-factor\footnote{{That is, vertex-disjoint copies of $K_3$ covering all the vertices.}} appears in a given $(n,d,\lambda)$-graph 
{has been a prominent question and is an instructive insight into the behaviour of pseudorandom graphs.} 
It is easy to infer from the expander
mixing lemma that if $\lambda\le 0.1 d^2/n$, then any
$(n,d,\lambda)$-graph contains a triangle (in fact, every vertex lies
in a triangle). An ingenious construction of Alon~\cite{Alon94} provides an example of a triangle-free $(n,d,\lambda)$-graph with
 $\lambda=\Theta(n^{1/3})$ and $d=\Theta(n^{2/3})$, which is essentially as dense as possible, considering the previous comments. 
This example can be bootstrapped, as is done in~\cite{KSS04}, to the whole possible 
range of $d=d(n)${, giving $K_3$-free $(n,d,\lambda)$-graphs with $\lambda= \Theta(d^2/n)$}.
Further examples of (near) optimal dense pseudorandom triangle-free graphs have since been given \cite{Kopparty11,conlon2017sequence}. 
On the other hand, Krivelevich, Sudakov and
Szab\'o~\cite{KSS04} proved that $(n,d,\lambda)$-graphs with
$\lambda=o\left(d^3/(n^2\log n)\right)$ contain a triangle-factor
if~$3\mid n$ and they made the following intriguing conjecture, which is one of the
 central problems in the theory of spanning structures in $(n,d,\lambda)$-graphs.
\begin{conjecture}[Conjecture~7.1 in~\cite{KSS04}]
  \label{conj:KSS}
  There exists an absolute constant $c>0$ such that if
  $\lambda\le cd^2/n$, then every $(n,d,\lambda)$-graph~$G$ on
  $n\in 3\mathbb{N}$ vertices has a triangle-factor.
\end{conjecture}
This conjecture is supported by their  result \cite{KSS04} that $\lambda\le 0.1 d^2/n$ implies the existence of  a fractional triangle-factor.
Furthermore, a recent result of three of the authors~\cite{HKP18a, HKP18b}  states that, under the condition
$\lambda\le (1/600) d^2/n$, any $(n,d,\lambda)$-graph~$G$ with~$n$
sufficiently large contains a family of vertex-disjoint triangles covering all
but at most $n^{647/648}$ vertices of~$G$, thus a `near-perfect' triangle-factor. A very recent, remarkable result of 
Nenadov~\cite{Nen18} 
infers that  $\lambda\leq cd^2/(n\log n)$ for some
constant~$c>0$ is sufficient to yield a triangle-factor. { Considering the triangle-free constructions mentioned above, we see that Nenadov's result is within a $\log$ factor of the optimal conjectured bound.} 
 Nenadov also raised the question in~\cite{Nen18} of whether a similar condition 
 would imply {the existence of} any {given} $2$-factor\footnote{{A $2$-factor is a 2-regular spanning subgraph.}} {in a pseudorandom graph}. {The purpose of this work is to give a positive answer to Nenadov's question, casting the question in terms of 2-universality and showing that we can efficiently find a given subgraph of maximum degree 2 in polynomial time.}

{In order to state our result we will switch\footnote{{Nenadov also worked in this broader class of pseudorandom graphs.}} to working with } 
$(p,\lambda)$-bijumbled graphs {(}introduced in~\cite{KRSSS07}{)}, 
which give a convenient, slight variant of Thomason's jumbledness. 
Bijumbled graphs  $G$ satisfy the property: 
\begin{equation} \label{eq:EML2}
\left|e(A,B)-p|A||B|\right|<\lambda\sqrt{|A||B|}
\end{equation}
for all $A$, $B\subseteq V(G)$. In particular it 
{is easy to see by the expander mixing lemma \eqref{eq:EML} that }
an $(n,d,\lambda)$-graph is $(d/n,\lambda)$-(bi)jumbled. Moreover, {the two concepts are closely linked as}
 a $({p},\lambda)$-(bi)jumbled graph is almost $pn$-regular{, in that almost all vertices have degree close to $pn$}.
 
{Before the current paper, the best result towards 2-universality in pseudorandom graphs is due to} Allen, B\"ottcher, H\`an and two of the authors 
~\cite{ABHKP17}{. There, they proved that there exists an $\eps>0$ such} that 
 $(p,\eps p^{5/2}n)$-bijumbled graphs of minimum degree $\Omega(pn)$  contain a square\footnote{A square of a graph $H$ is obtained by connecting its vertices at distance at most two through edges.
   The existence of a square of a Hamilton cycle implies $2$-universality as one can greedily find vertex-disjoint cycles of arbitrary lengths (see e.g. \cite{fan1996hamiltonian}). } of a Hamilton cycle {and hence are $2$-universal}. The proof is algorithmic, leading to an efficient procedure.  {Here we weaken the requirement on $\lambda$ to match that of Nenadov and obtain the following.}
 
 \begin{theorem}\label{thm:main}
{For all $\delta>0$, t}here exist constants $\eps>0$ and $n_0$ such that, for any $p\in(0,1]$, the following holds. For any $n\ge n_0$ and any given potential $2$-factor $F$ (that is, family of disjoint cycles whose lengths sum up to $n$),
 there is a polynomial time algorithm which finds a copy of $F$ 
 in any $(p,\lambda)$-bijumbled graph $G$ on $n$ vertices with $\lambda\le \eps p^2n/\log n$ 
and minimum degree $\delta(G)\ge {\delta}pn$.
\end{theorem}
 
In particular, Theorem \ref{thm:main} implies that such a $(p,\lambda)$-bijumbled graph $G$ is $2$-universal. Indeed, given a graph  $F'$ on  at most $n$ vertices with $\Delta(F')\le 2$, we find a supergraph $F$ of $F'$ on $n$ vertices, so that all but at most one of the components of $F$ are cycles. It is possible that $F$ may have either one isolated vertex or a single edge but since we can easily embed a single vertex/edge into a bijumbled graph $G$ altering its minimum degree only a little, it suffices to concentrate on the case that $F$ is a $2$-factor.  

We remark that the minimum degree condition in Theorem \ref{thm:main} is weak and natural. Indeed some minimum degree condition is necessary as otherwise one could have isolated vertices and the bijumbled definition \eqref{eq:EML2} guarantees that almost all vertices satisfy the minimum degree condition in any case. Finally, we mention that we do not try to optimise the running time of our algorithm and are satisfied with being able to provide a deterministic algorithm which is efficient in that it runs in polynomial time. Indeed, the problem of establishing the existence of certain 2-factors (e.g. for triangle-factors  \cite{kirkpatrick1983complexity} and Hamilton cycles \cite{karp1972reducibility}) in graphs is known to be NP-complete  and many proofs of existence of spanning structures in certain graph classes adopt probabilistic methods. 
 
 \subsection{Proof method}

 {Our proof uses} the absorption-reservoir method, {which has been a powerful tool in proving the existence of certain spanning structures and is }often superior to the aforementioned blow-up lemmas. The basic idea of the method is to carefully define an `absorbing structure' which can contribute to the desired spanning structure in many ways. One then finds such an absorbing structure in the host (hyper-)graph and, after putting this to one side, finds almost all of the desired spanning structure in the remainder of the host graph. The absorbing structure then provides the flexibility to `clean up' and complete the  spanning structure. 
The beginnings of this method date back to the early 90s, but the breakthrough {in}
 the wide applicability {of these methods, however, was} 
 first established by R\"odl, Ruci\'nski and Szemer\'edi~\cite{RRSz06,RRSz08} 
 in their study of Hamiltonicity in hypergraphs. 
 {There, the method was used to study dense hypergraphs but the methods have since been adapted to other settings (see e.g.~\cite{KO12, ABKP15, ABHKP17}).}  

In his work on spanning trees in random graphs~\cite{M14a,M19}, Montgomery ingeniously wove sparse `robust' bipartite graphs (which we call \emph{sparse templates}) 
into the absorption-reservoir method. 
The first use of sparse templates for the absorption in the context of pseudorandom graphs was recently given by the current authors in~\cite{HKMP18}.
 Here, we again use this idea and introduce for the first time, an efficient version of this new type of absorption, which may be of independent interest. In order to explicitly generate a sparse template we use bounded degree bipartite graphs with strong expansion properties. Such graphs are known as \emph{concentrators}~\cite{ACKRRS01,HLW06}. 

Our general proof approach here builds on the ideas from our paper~\cite{HKMP18}, derandomising additional arguments at various places and adapting the method to handle different types of $2$-factors. 
In order to deal with triangles we also require a different
absorbing-type argument, namely, the argument due to
Nenadov~\cite{Nen18}, for which we replace certain nonalgorithmic
arguments.

 \section{Proof of Theorem~\ref{thm:main}} \label{Sec: Proof overview}

The following three theorems will establish our main result. 
Note that the non-algorithmic version of Theorem \ref{thm:Nenadov} was proved in \cite{Nen18}.

\begin{theorem}[Theorem 1.2, \cite{Nen18}]\label{thm:Nenadov}
For every $\delta>0$ there exists a constant $\eps>0$ such that, for any $p\in(0,1]$, a $(p, \lambda)$-bijumbled graph $G$ on $n\in 3\NN$ vertices with $\lambda\le \eps p^2n/\log n$ and  minimum degree $\delta(G)\ge \delta pn $ contains a triangle-factor, which can be found with a deterministic polynomial time algorithm.
\end{theorem}

\begin{theorem}\label{thm:const_cycles}
For every $\delta>0$ and $L\in \NN$ there exist constants $\eps_0=\eps_0(\delta,L)>0$ and $n_0=n_0(\delta,L)$ such that for any $0<\eps<\eps_0$ the following holds.
Let $G$ be a $(p, \lambda)$-bijumbled graph on $n\ge n_0$ vertices with $p\in(0,1/2]$, $\lambda\le \eps p^2n$ and minimum degree $\delta(G)\ge \delta pn $.
Then in polynomial time, one can find any family of vertex-disjoint cycles with lengths in the interval $[4,L]$ whose lengths sum up to at most $n$.
\end{theorem}
\begin{theorem}\label{thm:long_cycles}
For every $\delta>0$ there exist constants $L \in \NN$, $\eps_1>0$ and $n_0$ such that the following holds. 
For any $p\in(0,1/3]$ and $0<\eps<\eps_1$, let $G$ be a $(p, \lambda)$-bijumbled graph on $n\ge n_0$ vertices with $\lambda\le \eps p^2n$ and  minimum degree $\delta(G)\ge \delta pn $.
Then in polynomial time, one can find any family of vertex-disjoint cycles with lengths in the interval $[L+1,n]$ whose lengths sum up to at most $n$.
\end{theorem}

Now we can quickly derive Theorem~\ref{thm:main}.

\begin{proof}[Proof of Theorem~\ref{thm:main}]
We consider three (not mutually exclusive) cases:
\begin{enumerate}[label=($\roman*$)]
\item there is subset of at least $n/2$ vertices of $F$ which induce a collection of vertex-disjoint triangles in $F$,
\item there is a subset of at least $n/4$ vertices of $F$ which induce a collection of vertex-disjoint cycles  with lengths in the interval $[4,L]$ where $L$ is some absolute constant determined by Theorem \ref{thm:long_cycles} above, 
\item there is a vertex subset of at least $n/4$ vertices of $F$ which induce a collection of vertex-disjoint cycles with lengths in the interval $[L+1,n]$.
\end{enumerate}

For a given $2$-factor $F$ on at most  $n$ vertices, we are in one of the three cases defined above. Let $F_1$, $F_2$ and $F_3$ denote the subgraphs of $F$, so that all triangles constitute $F_1$, all cycles with lengths in $[4,L]$ constitute the subfamily $F_2$ and all cycles of length at least $L+1$ are $F_3$. We set 
 $n_i:=v(F_i)$ for each $i\in [3]$. 

If we are in the first case ($n_1\ge n/2$) we  partition the vertex set $V$ of $G$ into three parts $V_1\dcup V_2\dcup V_3$, so that $|V_1|=n/2$ and $|V_3|=|V_4|=n/4$ and each $G[V_i]$ remains a $(p, \lambda)$-bijumbled graph. Moreover,  every vertex $v\in V$ satisfies $\deg(v,V_i)\ge\delta p|V_i|/2$ for any $i\in[3]$. Clearly, 
one could achieve this via a random partition and it will be possible to derandomise this approach (see  Corollary~\ref{cor:partition}).  
If $n_2\ge n_3$, then we 
first apply Theorem~\ref{thm:long_cycles} to embed $F_3$ via some embedding $\varphi_3$  into $G[V_3]$. Then Theorem~\ref{thm:const_cycles} asserts that $F_2$ can be embedded into $G_2':=G[(V_2\dcup V_3)\setminus \varphi_3(V(F_3))]$, since $G'_2$ is itself a $(p, \lambda)$-bijumbled graph with minimum degree at least  $\delta pn/4$. Finally, we apply Theorem~\ref{thm:Nenadov} to embed $F_1$ into the remaining graph (which is again $(p, \lambda)$-bijumbled graph with minimum degree at least $\delta pn/2$). If $n_3\ge n_2$ then we first embed $F_2$, then $F_3$ and, finally, $F_1$. The other cases $n_2\ge n/4$ and $n_3\ge n/4$ are treated analogously. 
\end{proof}

\subsection{Structure of the paper}
It remains to prove Theorems~\ref{thm:Nenadov} -- ~\ref{thm:long_cycles}.
We will only consider the case $p\le 1/3$, since the dense case can be treated fairly easily by the algorithmic version of the blow-up lemma due to Koml\'os,
S\'ark\"ozy and Szemer\'edi~\cite{KSSz98}.   
In Section \ref{Sec: Auxiliary results} we collect some notation and useful tools and algorithms for our study. 
 In the subsequent two sections we prove the first two theorems (Theorems~\ref{thm:const_cycles} and~\ref{thm:long_cycles}) and in  Section~\ref{Sec: Proof of Theorem Nenadov} we replace one non-algorithmic argument from \cite{Nen18} with a constructive proof. 
 
 Throughout we use the shorthand $(p,\lambda)$-graphs to refer to $(p,\lambda)$-bijumbled graphs, we write $\log$ to denote the natural logarithm and we omit floor and ceiling signs in order not to clutter the arguments. The final section closes with some problems left for further study.

\section{Auxiliary results} \label{Sec: Auxiliary results}

\subsection{Simple statements about $(p, \lambda)$-bijumbled graphs}
 In this section, we collect some useful properties of $(p, \lambda)$-graphs. We will use the following notation. Given a graph $G=(V,E)$, we denote by $\deg(v,U)$ the number of neighbours of $v\in V$ in $U\subseteq V$. A $u$-$v$-path  is a path $P$ with end vertices $u$ and $v$, and we call the other vertices of  $P$ {the} \emph{inner vertices}. For vertex subsets $A,B$, an $A$-$B$ path is a $u$-$v$ path for some vertices $u\in A$ and $v \in B$. The length of a path is the number of its edges.  Finally we will denote by $\Clm$, the graph that consists of a path $P$ of length $\ell-2$, whose end vertices have exactly $K$ distinct common neighbours outside of $V(P)$, for some $K \in \NN$. 
We start with the following remark which follows directly from the definition (\ref{eq:EML2}).

\begin{remark} \label{rem:easy bijumbled}
If $\eps>0$ and $A$ and $B$ are subsets of a $(p, \lambda)$-graph with $\lambda \leq \eps p^2n$, such that $|A||B|\ge 4\eps^2p^2n^2$, then $e(A,B)\ge \frac{p|A||B|}{2}$.
\end{remark}

Next, we show that bijumbled graphs cannot be too sparse.

\begin{proposition}\label{prop:lower_bd}
Given $\eps\in(0,1)$, there exists $n_0\in \NN$ such that 
if $G=(V,E)$ is a $(p, \lambda)$-graph on $n\geq n_0$ vertices with $\lambda\le \eps p^2n$ {and $\eps p\le 1/2$}, then 
$p\ge (\eps^2n)^{-1/3}/4$.
\end{proposition} 
\begin{proof}
Let $S\subseteq V$ be a set of at least $n/2$ vertices. 
Then there is a vertex $v\in S$ whose degree in $G[S]$ is at most $2p|S|$. 
Indeed, we have $\sum_{v\in S} \deg(v,S)=2e(G[S])\le p|S|^2+\lambda|S|$, which implies that the average degree in $G[S]$ is at most $p|S|+\lambda\le 2p|S|\le 2p n$.

We consecutively find vertices $v_1,\dots, v_t$ with $t= n/(2+4pn)$ such that setting $V_i:=V\setminus\{v_1,\ldots,v_{i-1}\}$, we have $\deg(v_i, V_i)\leq 2p |V_i|\leq 2pn$. Thus
 setting $U:=\{v_i\colon i\in[t]\}$ and $W:=V\setminus \left(U\cup\bigcup_{i\in [t]} N(v_i)\right)$, we have that $|W|\geq n-t(1+2pn) = n/2$ and 
$e_G(U,W)=0\ge p|U||W|-\lambda\sqrt{|U||W|}$. 

It follows that $\eps p^2 n\ge \lambda \ge p\sqrt{t n/2}$. Thus, $2\eps^2 p^2  n\ge t= n/(2+4pn)$, which implies $p^2\ge (\eps^{-2}/4)\min\{1/(4pn),1/2\}$. Rearranging we get $p\ge \min\{(\eps^2n)^{-1/3}/3,1/(2\sqrt{2}\eps)\}\ge (\eps^2n)^{-1/3}/4$ for $n$ sufficiently large.
\end{proof}

The following fact also concerns the edge distribution of bijumbled graphs.

 \begin{fact}\label{fact:irregular} 
  Let $\eps>0$ and $G$ be a $(p, \lambda)$-graph on $n$ vertices with $p\in(0,1]$
   and $\lambda\le \eps p^2n$.
   \begin{enumerate}[label=$(\roman*)$]
   \item If $U$ is a set of vertices, then there are at most $4\eps^2p^2n^2/|U|$ vertices $w$ in $G$ with $|N_G(w)\cap U| < p |U|/2$. \label{fact:irre1}
  \item Given an integer $t$ and vertex sets $U_1,\dots, U_t, W$ such that $|W| > \sum_{i=1}^t 4\eps^2p^2n^2/|U_i|$, we can find a vertex $w\in W$ such that $|N_G(w)\cap U_i| \ge p |U_i|/2$ for all $i\in [t]$, in polynomial time.
  \label{fact:irre2}
  \end{enumerate}
 \end{fact}
 \begin{proof}
 Let $U'$ be the set of vertices $w$ such that
  $|N_G(w)\cap U| < p |U|/2$. From~\eqref{eq:EML2} 
  we have $|U'|p|U|/2> e(U',U)\ge p|U||U'|-\lambda\sqrt{|U||U'|}$. The conclusion follows from rearranging.
  
By the first part of the fact, $W$ clearly contains a desired vertex.
We find it by screening the degree of any vertex of $W$ into each $U_i$, which takes polynomial time. 
 \end{proof}

Next, given three sets $A$, $B$ and $C$, we show how to find an $A$-$B$-path of given length such that the inner vertices are from $C$.

 \begin{proposition}\label{prop:find_paths}
  Let $\eps>0$, $\ell \in \NN$ and  $G$ be a $(p, \lambda)$-graph on $n$ vertices with $p\in(0,1]$,
   $\lambda\le \eps p^2n$ and $\eps pn\ge 1$. If $A$ and $B$ are sets of at least 
   $2^{\ell-1}\eps p n$ vertices and $C$ is a set of at least 
   $2^{\ell-1}\eps n$ vertices, then in polynomial time, we can find an $A$-$B$-path $P$ of length $\ell$ whose inner vertices lie in $C$. 
 \end{proposition}
 \begin{proof}
  If $\ell=1$ then we have $e(A,B)> p|A||B|-\lambda\sqrt{|A||B|}\ge \sqrt{|A||B|}\left(p\eps p n-\lambda\right){\ge} 0$, namely, there is an edge with one end in $A$ and the other in $B$.
  We can find such an edge by searching the neighbourhoods of vertices in $A$ one by one. 
    We proceed now inductively and we assume that $\ell\ge 2$ and the assumption holds for $\ell-1$. 
   
  By Fact~\ref{fact:irregular}~\ref{fact:irre2} we find a vertex $a\in A$ with degree 
  at least $p|C|/2$ into $C$  in polynomial time. 
 Applying our inductive hypothesis to $N(a)\cap C$, $B\setminus \{a\}$ and $C\setminus\{a\}$ we find an $(N(a)\cap C)$-$(B\setminus\{a\})$-path of length $\ell-1$ with inner vertices in $C$, which together with $a$ yields the desired path of length~$\ell$.
 \end{proof}

We will use copies of $\Clm$ in our absorbing structure. The following simple fact asserts that we can find these copies in any large enough set of vertices.

 \begin{fact}\label{fact:cyclecount}
  Let $\eps>0$, $K\in \NN$ and let $G$ be a $(p, \lambda)$-graph on $n$ vertices with $p\in(0,1]$
   and $\lambda\le \eps p^2n$. Let $\eps p^2n\ge K/4$, $\ell\ge 4$ and   $U$ be a set of at least 
   $2^{\ell}\eps n$ vertices. Then we can find a copy of $\Clm$, and thus also a copy of $C_\ell$ in $U$, in polynomial time.
 \end{fact}
 \begin{proof}
Let $U'_1$ be the set of  vertices $v\in U$ with $|N(v)\cap U|< p|U|/2$. 
Since $|U|\ge  2^\ell \eps n$, Fact~\ref{fact:irregular}~\ref{fact:irre1} implies that $|U'_1|\le \eps p^2n/4$. 
We fix a vertex $u_1\in U\setminus U'_1$, i.e.\ $\deg(u_1,U)\ge p|U|/2$. 
Let $U'_2$ be the set of  vertices $v\in U$ with  $|N(v)\cap (N(u_1)\cap  U)|<p|N(u_1)\cap  U|/2$. 
Since $|N(u_1)\cap  U|\ge p|U|/2\ge 8 \eps pn$, Fact~\ref{fact:irregular}~\ref{fact:irre1} implies that $|U'_2|\le \eps pn/2$.
Thus, we have $|U_1'\cup U'_2|\le \eps p n$.
 
 We choose an arbitrary vertex $u_2\in U\setminus (U'_1\cup U'_2\cup\{u_1\})$. If $\ell=4$ then we clearly find a copy of $C_\ell(1,1,1,K)$ in $U$, because $|N(u_1)\cap N(u_2)\cap  U|\ge p^2|U|/4 \ge K+1$.
 If $\ell\ge 5$, then we first set aside a set $W$ of $K$ vertices from the common neighbourhood of $u_1$ and $u_2$. Now due to the fact that $|(N(u_i)\cap U)\setminus(W\cup \{u_1,u_2\})|\ge 2^{\ell-2}\eps p n$ for $i=1,2$ and 
 $|U\setminus (W\cup\{u_1,u_2\})|\ge 2^{\ell-2}\eps n$, we find by Proposition~\ref{prop:find_paths} 
 a path of length $\ell-4$ between $N(u_1)\cap U$ and $N(u_2)\cap U$, which together with $u_1$, $u_2$ and $W$, forms a copy of $\Clm$.
 
For the running time, by Fact~\ref{fact:irregular}~\ref{fact:irre2}, we can find $u_1$ and $u_2$ in polynomial time and 
 the rest of the proof runs in polynomial time  
 because we use Proposition~\ref{prop:find_paths}.
 \end{proof}
  
The following lemma asserts that we can (greedily) find almost spanning paths in $(p, \lambda)$-graphs. 
\begin{lemma}\label{lem:long_paths}
 Let $\eps>0$ and~$G$ be a $(p, \lambda)$-graph on $n$ vertices with $p\in(0,1/2]$ and
  $\lambda \le \eps p^{2}n$. If $U$ is a vertex subset of size greater than $\eps n$, then 
  we can find any path of length $\ell \le |U|-\eps n$ in $U$ in polynomial time. 
\end{lemma}
\begin{proof}
By Fact~\ref{fact:irregular} there is a vertex $u\in U$ of degree at least $p|U|/2$ in $U$. 
This gives us a path of length $0$. 
Assume now that we found inductively a path $P_t=u_0u_1\ldots u_t$ of length $t\le \lfloor |U|-\eps n\rfloor-1$ such that $\deg(u_t,U\setminus V(P_t))\ge p|U\setminus V(P_t)|/2$. Then by Fact~\ref{fact:irregular}~\ref{fact:irre1}, as $|U\setminus V(P_t)|\ge \eps n$
 and using that $p\le 1/2$, 
we have that there exists a vertex $u_{t+1}\in N(u_t)\cap (U\setminus V(P_t))$ with $\deg(u_{t+1},U\setminus V(P_t))\ge p|U\setminus V(P_t)|/2$ and the induction step is complete.

Since the proof is a repeated application of Fact~\ref{fact:irregular}~\ref{fact:irre1}, Fact~\ref{fact:irregular}~\ref{fact:irre2} implies that the running time is polynomial. 
\end{proof}
\subsection{Partitioning vertex sets}
At various points in our proof, we will wish to partition our vertex set in such a way that every vertex maintains good degree to all parts of the partition. 
This can be easily achieved probabilistically by choosing a random partition. However this idea can {also} be derandomised and achieved computationally efficiently. We use the following theorem of Alon and Spencer. 
\begin{theorem}[Theorem 16.1.2 in \cite{alon16}] \label{thm:derandomised partition}
Let $(a_{ij})_{i,j=1}^n$ be an $n \times n$ $0/1$-matrix. Then one can find, in polynomial time, $\varepsilon_1,\ldots, \varepsilon_n\in \{-1,1\}$ such that for every $1 \le i \le n$, it holds that $|\sum_{j=1}^n\varepsilon_ja_{ij}|\le \sqrt{2n\log(2n)}$. 
\end{theorem}

  \begin{corollary} \label{cor:partition}
Let $k\in \NN$ $\eps, \beta, \delta>0$ and $p\in (0,1]$. Then there exists $n_0 \in \NN$ such that for any $(p, \lambda)$-graph $G$ on $n\ge n_0$ vertices such that $\lambda \leq \eps p^2n$, the following holds. 
Let $U,W\subset V(G)$ be subsets of vertices such that $|U|\ge\beta n$ and for all $w\in W$, $\deg(w,U)\ge \delta p|U|$. 
Then in polynomial time, we can find $s:=2^k$ sets $U_1,\ldots, U_s \subseteq U$ such that $U=U_1 \dcup \ldots \dcup U_s$ forms an equipartition\footnote{Due to divisibility constraints, we formally mean here that the the sizes of the sets differ by at most one, and so $|U_i|\in \{\lfloor |U|/s\rfloor, \lceil |U|/s\rceil \}$ for each $i$.} of $U$ 
 and for all $w\in W$ and $i \in [s]$, $\deg(w, U_i) \ge \delta p|U_i|/2$. 
 \end{corollary}
 
 \begin{proof}
 We apply Theorem~\ref{thm:derandomised partition} to the adjacency matrix of $G$, where we add an all one row and an all one column and impose that row $i$ is all zero if $i \notin W$ and column $j$ is all zero if $j \notin U$. We let $U_{b}'=\{j \in U: \varepsilon_j=(-1)^b\}$, for $b=1,2$. The last row of the matrix guarantees that $||U_1'|-|U_2'||\le \sqrt{2(n+1)\log(2n+1)}=:g(n)$. The other rows guarantee that the vertices in $W$ have good degree to both sets, so that after moving some vertices from one of the sets to another in order to balance $|U_1'|$ and $|U_2'|$, we have that for all $w\in W$, $\deg(w, U_i') \ge \delta p|U|/2- 2g(n)$. 
 
We can now apply the above procedure to each $U_i'$, with the new minimum degrees. Repeating this $k$ times, we end up with $U_1,\ldots, U_s$ as an equipartition of $U$ such that for any $w\in W$, $\deg(w,U_i) \ge \delta p |U|/s - 2kg(n).$
Owing to Proposition \ref{prop:lower_bd}, we are done because for sufficiently large $n$, $2kg(n) \le  \delta\beta pn/(2s) \le \delta p|U|/(2s) $.
 \end{proof}

\subsection{A connecting lemma}

The lemma below allows us to close many paths (whose ends are `well-connected' into a large set) into cycles using short paths of a fixed prescribed length. In the following lemma a $v$-$v$-path refers to a cycle through $v$ whose inner vertices are all the vertices of the cycle not equal to $v$.
\begin{lemma}[Multiple connection lemma]\label{lem:connections}
  For every $0<\beta,\delta'\le 1$, $\ell\ge 3$ there exists $\eps_0>0$ and $n_0\in \NN$ such that for all $\eps\in (0,\eps_0)$  and $n\ge n_0$ the following holds.  Let~$G$ be a $(p, \lambda)$-graph on $n$ vertices with $p\in(0,1]$ and
  $\lambda \le \eps p^{2}n$.  
  Let  $U$ be  a  vertex subset of size at least $\beta n$ and $(a_i,b_i)_{i\in[r]}$ a system of pairs of vertices in $G$, so that  every vertex occurs at most twice in $(a_1,\ldots,a_r,b_1,\ldots,b_r)$ and
   $U$ is   disjoint from $\bigcup_i\{a_i,b_i\}$. 
   If  $r\le |U|/(8\ell)$ and 
   $\deg(a_i,U)$, $\deg(b_i,U)\ge \delta' p |U|$ for  all $i\in[r]$ then the following holds.
  In polynomial time, we can find a family $\cQ$ of length $\ell$ $a_i$-$b_i$-paths $Q_i$, 
  whose inner vertices are pairwise disjoint and lie in $U$. 
\end{lemma}
\begin{proof}
Fix $\eps_0\le \delta' \beta 2^{-(\ell+6)}/\ell$. Firstly, using Corollary \ref{cor:partition}, in polynomial time, we can split $U$ into $U=U_1 \dcup U_2$ such that $|U_1|=|U_2|=|U|/2$ and $\deg(a_i,U_b),\deg(b_i,U_b)\geq \delta' p|U|/4$ for all $i$ and $b=1,2$. We will build our paths algorithmically in two phases, first using vertices of $U_1$ and then vertices of $U_2$. We initiate by letting $\cQ'=\emptyset$, $U_1'=U_1$ and $U_2'=U_2$. We will use $\cQ'$ to denote our intermediate family of paths and $U_1',U_2'$ the remaining sets of vertices that we can use. Note that throughout we will have $|V(\cQ)|\leq r\ell\leq |U|/8$, and thus $|U_1'|,|U_2'|$ will always have size at least $|U|/4$. 

We proceed as follows. If there is an $i\in [r]$ such that $\deg(a_i,U_1'),\deg(b_i,U_1')\ge \delta'p|U|/8$, then using Proposition \ref{prop:find_paths}, in polynomial time  
we find a length $\ell-2$ path $P_i$ from a vertex in $N(a_i)\cap U_1'$ to a vertex in $N(b_i)\cap U_1'$ using vertices in $U_1'$. 
Add $Q_i:=a_i$-$P_i$-$b_i$ to $\cQ$ and delete the vertices of $P_i$ from $U_1'$. At the end of this phase, let $I\subseteq [r]$ be the remaining indices. 
Since each vertex appears at most twice in $(a_i,b_i)_{i\in[r]}$, by Fact \ref{fact:irregular}~\ref{fact:irre1}, we have that
\[
|I| \leq \frac{4\eps^2 p^2 n^2}{|U_1'|}\le \frac{4\eps^2p^2n^2}{\beta n/4} \le 16 \eps_0 p^2n \leq \delta'p^2|U|/(8\ell) \le \delta'p|U|/(8\ell),
\] 
where we used $|U_1'|\ge |U|/4 \ge \beta n/4$, and $\eps \le \eps_0\le \delta' \beta 2^{-(\ell+6)}/\ell$.
Now we run the process again, using $U_2$ in place of $U_1$. As $|V(\cQ')\cap U_2|\leq \delta'p|U|/8$ throughout, we can proceed greedily by the degree assumptions and complete the family $\cQ$. 
Note that in each step, we need to screen the degrees of the remaining pair $a_i$ and $b_i$. 
The application of Proposition~\ref{prop:find_paths} then  runs in polynomial time  and so 
the whole algorithm runs in polynomial time.
\end{proof}


\subsection{An explicit template} \label{subsec: template}

A \emph{template} $T$ with \emph{flexibility} $m\in \mathbb{N}$ is a
bipartite graph on $7m$ vertices with vertex parts $I$ and
$J=J_1 \dcup J_2$, such that $|I|=3m$, $|J_1|=|J_2|=2m$, and for any
$\bar{J}\subset J_1$, with $|\bar{J}|=m$, the induced graph
$T[V(T)\setminus\bar{J}]$ has a perfect matching. We call $J_1$ the
\textit{flexible set} of vertices for the template. 

Sparse templates, with maximum degree smaller than some absolute constant, are very useful in absorption arguments and can be used to design robust absorbing structures. 
 Montgomery first introduced the use of such templates when applying the absorbing method in his work on spanning trees in random graphs~\cite{M14a,M19}. Ferber, Kronenberg and Luh \cite{FKL16} followed the same argument as Montgomery (with some small adjustments) when studying the 2-universality of the random graph. Kwan \cite{Kwan16} also used sparse templates  to study 
random Steiner triple systems, generalising the template to a hypergraph setting and using it to define an absorbing structure for perfect matchings. 
Further applications were given by Ferber and Nenadov~\cite{FN17} in their work on universality in the random graph, recently by the current authors in~\cite{HKMP18} which was the first use of the method in the context of pseudorandom graphs,
  and by Nenadov and Pehova \cite{nenadov2018ramsey} who used the method to study a variant of the Hajnal-Szem\'eredi Theorem. 
The final three papers mentioned all adapt the method to give absorbing structures which output disjoint copies of a fixed graph $H$ (a partial $H$-factor), however the different absorbing structures used are interestingly all  significantly distinct. 

It is not difficult to prove the existence of sparse
templates for large enough~$m$ probabilistically; see
e.g.~\cite[Lemma~2.8]{M14a}. As we wish to give a completely algorithmic proof, in this section we show how to build a template $T$ efficiently.
We use the following result of Lubotzky, Phillips and Sarnak~\cite{LPS88}.
\begin{theorem}\cite{LPS88}\label{thm:raman}
For primes $p, q\equiv1 \pmod 4$ such that $p$ is 
a quadratic residue modulo $q$, one can construct an explicit $(p+1)$-regular Ramanujan graph $G$ in polynomial time (in $q$) {with $(q^3-q)/2$ vertices}.
\end{theorem}

{A Ramanujan graph, by definition, is a $d$-regular graph all of whose  eigenvalues (other than $d$ and, if bipartite, $-d$) are in absolute value at most $2\sqrt{d-1}$.} We {will} in fact 
use a bipartite Ramanujan graph {constructed as follows}.
{Consider} the graph $G$ provided by Theorem~\ref{thm:raman} -- take $V_1$ and $V_2$ as two identical copies of $V(G)$, and join $v_1\in V_1$ and $v_2\in V_2$ if and only if the preimages of $v_1$ and $v_2$ in $V(G)$ form an edge of $G$. {It is clear that this bipartite Ramanujan graph is still $d$-regular and satisfies the expander mixing lemma \eqref{eq:EML} for all $A\subseteq V_1$ and $B\subseteq V_2$, where $n$ is the number of vertices in each part, and $\lambda=2\sqrt{d-1}$.}

\begin{proposition}\label{prop: perfect matching}
Let $d\ge 144/\alpha^2$.
Let $G$ be a {bipartite $d$-regular Ramanujan graph on vertex set $V_1 \dcup V_2$, with $|V_1|=|V_2|=n$.}
Suppose $U{\subseteq V_1}$ and $W{\subseteq V_2} $ are 
vertex subsets of $V(G)$ such that $|U|=|W|=\alpha n$  and $\deg(w, U)\ge \alpha d/3$ for any $w\in W$ and $\deg(u, W)\ge \alpha d/3$ for any $u\in U$.
Then $G[U, W]$ contains a perfect matching.
\end{proposition}

\begin{proof}
We will verify Hall's condition for $G[U, W]$.
First we claim that it suffices to prove the condition $|N(X)\cap W|\ge |X|$, for sets $X\subset U$ of size $|X|\le \lceil|U|/2\rceil=\lceil\alpha n/2\rceil$. Indeed by symmetry, one can then conclude that Hall's condition holds for subsets of $W$ which have size smaller than $\lceil \alpha n/2 \rceil $. Now for $X\subseteq W$ such that $|X|> \alpha n/2$, if $|N(X)\cap W|<|X|$, then setting $Y'$ to be a subset of $W \setminus N(X)$ of size $\alpha n- |X|+1$ we have that $|Y'|\le \lceil \alpha n/2 \rceil $ and so from above we can conclude that $|N(Y')\cap U|\geq |Y'|$. This contradicts the definition of $Y'$ as we must have that $N(Y')$ intersects $X$ and hence $Y'$ intersects $N(X)$.

So it remains to prove that, for $X\subset U$ with $|X|\le \lceil \alpha n/2 \rceil$, taking $Y=N(X)\cap W$  we have that $|Y|\ge |X|$.
Assume to the contrary that $|Y|< |X|$.
We first assume that $|X|\le \alpha n/6$.
By the degree condition, we obtain that $e(X, Y) \ge |X| \alpha d/3$.
On the other hand, by~\eqref{eq:EML}, we have
\[
e(X, Y)\le \frac dn |X| |Y| + \lambda \sqrt{|X| |Y|} < \frac{\alpha d}{6} |X| + 2\sqrt{d|X||Y|}.
\]
Putting these together, we get $2\sqrt{d|X||Y|} \ge \alpha d|X|/6$.
By $d\ge 144/\alpha^2$, this implies $|Y| \ge |X|$, a contradiction.
Next we assume that $\alpha n/6< |X|\le \lceil \alpha n/2 \rceil$.
By $|W\setminus Y|\ge \alpha n/2 $, ~\eqref{eq:EML} and the fact that  $\alpha^2 d\ge 144$ we have that
\[
e(X, W\setminus Y)\ge \left(\frac dn \sqrt{|X| |W\setminus Y|}- \lambda\right) \left(\sqrt{|X||W\setminus Y|}\right)> \alpha^2 d n/12 - 2\sqrt{d \alpha^2 n^2/12}>0,
\]
   contradicting the definition of $Y$. 
\end{proof}

\begin{lemma} \label{lem:bipartite template}
Let $p\equiv 1\pmod 4$ be a prime such that {$p \ge 68000$.}
For a sufficiently large integer $m$,  a template with flexibility $m$ and maximum degree ${d:=}p+1$ can be constructed in polynomial time.
\end{lemma}
\begin{proof}
{It follows from the} Siegel--Walfisz theorem~\cite{Wal36} {that} for large $x$, the density in $[x]$ of the set of primes which are $1\pmod{4p}$ is  of the order $\Theta(1/\log x)$.  Therefore for all  sufficiently large $m$ we can pick a prime $q\equiv {1\pmod{4p}}$ between 
{$(21m)^{1/3}$ and $1.01(21m)^{1/3}$}.
{Thus $20m\le q^3-q \le 22m$.}
{Using quadratic reciprocity to infer that $p$ is a quadratic residue modulo $q$, we have by} 
Theorem~\ref{thm:raman}
{that} we can construct in polynomial time a bipartite $d$-regular Ramanujan graph $G=(X\cup Y, E)$ with 
{$10m\le |X| =|Y|\le 11m$} 
and second eigenvalue $\lambda \le 2\sqrt{d}$.
We first show that for any set $U\subseteq X$ (or $Y$) of size at least $3m/2$, there are at most 
{$34000m/d$} 
vertices $v$ in $Y$ (or $X$) such that 
{$\deg(v, U)< d/10$}.
Indeed, denote by $B$ the set of such vertices $v$. Clearly we have 
{$e(U, B) < d |B|/10$}.
On the other hand, by~\eqref{eq:EML}, we have
{\[
\frac{d |B|}{10} > e(U, B) \ge \frac d{11m} |B| |U| - \lambda \sqrt{|B||U|} \ge \frac{3d|B|}{22} - 2\sqrt{d|B|\cdot 11m}.
\]
This implies that $2d|B|/55 < 2\sqrt{11d|B|m}$, that is, $|B| < 33275m/d<34000m/d$, as claimed.} 

Now take arbitrary sets $V_1''\subset X$, $V_2''\subset Y$ such that $|V_1''|=3m$ and $|V_2''|=2m$.
Next, we sequentially delete vertices from $V_1''$ and $V_2''$ as follows.
\begin{itemize}
\item Initiate with $V_i':=V_i''$ for $i=1,2$.
\item If there is a vertex $v\in V_1'$ such that $\deg(v, V_2') < 
{d/10}$,
then delete $v$ from $V_1'$, \label{item2}
\item If there is a vertex $v\in V_2'$ such that $\deg(v, V_1') < 
{d/10}
$, then delete $v$ from $V_2'$. \label{item3}
\end{itemize}
Note that since $|V_i''| - {34000}
m/d \ge 3m/2$, by our claim above, at most ${34000}
m/d$ vertices will be deleted from each set.
Denote by $V_1'$ and $V_2'$ the resulting sets.
Next, since there are at most $
{34000}
m/d$ vertices that have degree less than $
{d/10}
$ to $V_i'$, $i=1$, $2$, respectively, we can add vertices to $V_1'$ and $V_2'$ and obtain $V_1$ and $V_2$ such that $|V_1|=3m$, $|V_2|=2m$ and $\deg(v, V_i)\ge 
{d/10}
$ for any $v\in V_{3-i}$, $i=1,2$.
Finally, we pick $J_1$ as a set of $2m$ vertices in $Y\setminus V_2$ which have degree at least
{$d/10$}
to $V_1${.}


We claim that $T=G[V_1\cup V_2\cup J_1]$ is the desired template with flexible set $J_1$.
It remains to check the property of $T$.
For this, take any set $J'$ of $m$ vertices in $J_1$ and consider $G[V_1, V_2\cup J']$.
Since the assumptions of Proposition~\ref{prop: perfect matching} are satisfied with $\alpha = 3m/|X|
{\in[3/11, 3/10]}$
, $G[V_1, V_2\cup J']$ has a perfect matching and we are done.
For the running time, note that in each of the steps above, it is enough to query the neighbourhood of a vertex, which can be done in constant time.
So the overall running time is polynomial in $m$.
\end{proof}

\section{Proof of Theorem~\ref{thm:const_cycles}}  \label{Sec: Proof of Theorem constant cycles}

In~\cite{HKMP18} an \emph{absorbing structure} for cliques was defined. 
Here we generalise it for cycles as follows. Assume that 
$T=(I, J_1\cup J_2, E)$ is a bipartite template with flexibility
$m$, maximum degree $\Delta(T)\leq K$ and flexible set $J_1$.  
It will be convenient to identify $T$ with its edges which may be viewed as 
the corresponding subset of tuples $(i,j)\in [3m]\times[4m]$, hence we will also think of $I$ as $[3m]$, $J_1$ as $[2m]$, 
 $J_2$ as $[2m+1,4m]:=\{2m+1,\ldots, 4m\}$ and $J=J_1\cup J_2$.
  
 An \emph{absorbing structure for cycles of length $s+2$} is a tuple $\cS=(T, \cP_1, A, \cP_2, Z, Z_1)$ which consists 
of the template $T$ with flexibility $m$, 
the two sets $\cP_1$ and $\cP_2$ of
vertex-disjoint paths of fixed length $s$ and three vertex sets $A$, $Z$ and $Z_1$ with $Z_1\subset Z$. 
Furthermore, the sets  $V(\cP_1)$, $V(\cP_2)$, $A$ and $Z$ are pairwise disjoint and 
 with the labelling $Z_1=\{z_1,\dots, z_{2m}\}$, 
$Z_2=\{z_{2m+1}, \dots, z_{4m}\}$ (so that  $Z:=Z_1\cup Z_2$),
$\cP_1:=\{P^1, P^2, \dots, P^{3m}\}$,
$A=\{a_{ij}: (i,j)\in E(T)\}$ and
$\cP_2=\{P_{ij}: (i,j)\in E(T)\}$, the following holds in $G$ for   $(i,j)\in E(T)$:
\begin{itemize}
\item $a_{ij}$ is adjacent to the ends of $P^{i}$, i.e.\ closes a cycle on $s+2$ vertices,
\item each $a_{ij}$ is adjacent to the ends of $P_{ij}$,
\item each $z_{j}$ is adjacent to the ends of $P_{ij}$.
\end{itemize}

It is well known that finding  a maximum matching in  a graph can be done in polynomial time. Using this and unpacking the definition of the absorbing structure leads to the following fact. 

\begin{fact}\label{fact:absorbing}
  The absorbing structure $\cS=(T, \cP_1, A, \cP_2, Z, Z_1)$ has the property
  that, for any subset $\bar Z\subseteq Z_1$ with $|\bar Z|=m$, the
  removal of $\bar Z$ leaves a graph with a $C_{s+2}$-factor, which can be found in polynomial time.
\end{fact} 

\begin{proof}
  By the property of the template $T\subseteq [3m]\times [4m]$, there is a
  perfect matching~$M$ in $[3m]\times ([4m]\setminus \bar{J})\cap T$ with $\bar{J}:=\{j\colon z_j\in\bar{Z}\}$. Furthermore, we can find $M$ in polynomial time.
  
  
  Then for each edge $(i,j)\in M$, we take the $(s+2)$-cycles on
  $\{a_{ij}\}\cup {P}^{i}$ and $\{z_{j}\}\cup {P}_{ij}$; for the
  edges $(i,j)\in E(T)\setminus M$, we take the
  $(s+2)$-cycle on $\{a_{ij}\}\cup {P}_{ij}$.  This gives the desired
  $C_{s+2}$-factor.
\end{proof}

The following lemma is a variant of Lemma~2.7 from~\cite{HKMP18}.   
\begin{lemma}\label{lem:absorbing_structure}
  {Let $K:=68042$}. For every $\delta>0$, $\ell\ge 4$ and $\alpha\in (0,\alpha(\ell)]$ (where $\alpha(\ell):=1/(60\ell(K+2))$) there 
  exists $\eps_0>0$ such that for all $\eps\in (0,\eps_0)$ there
  is an $n_0\in \NN$ such that the following holds for all
  $n\ge n_0$.  Let~$G$ be a $(p, \lambda)$-graph with $n$ vertices, $p\in(0,1/3]$,
  $\lambda \le \eps p^{2}n$, $\delta(G) \geq \delta pn$ and suppose $m=\alpha n$. 
   Then in polynomial time we can find an absorbing structure $\cS=(T, \cP_1, A, \cP_2, Z, Z_1)$ for cycles of length $\ell$  with
  flexibility~$m$ in $G$. Further, one can find in polynomial time a set $W\subseteq V(G)\setminus V(\cS)$, with $|W|= n/4$ and $\deg(v,W)\geq \delta p |W|/8$ for all vertices $v$ of $G$. 
\end{lemma}  
\begin{proof}
  First we choose $\eps_0=\min\{\delta/(400K\ell),2^{-(\ell+6)},\alpha\}$ and let $\eps\in(0,\eps_0)$. Then
  we take $n_0$ large enough. Therefore, owing to Proposition~\ref{prop:lower_bd}, 
  quantities $p^2n$ and $pn$ are large as well.

We consider a partition of $V(G)=V_1\dcup V_2\dcup V_3\dcup V_4$ with $|V_1|=|V_2|=|V_3|=|V_4|=n/4$, such that 
\begin{equation}\label{eq:rand_mindeg}
\deg(v,V_i)\ge \delta p|V_i|/2
\end{equation}
 for all $i\in[4]$ and $v\in V$, as given by Corollary~\ref{cor:partition}. We fix $W=V_4$ and thus the conditions on $W$ are satisfied. We now build our absorbing structure using vertices of $V(G)
\setminus W$. Throughout the proof, we denote the intermediate partial absorbing structure by $\cS'$. Note that an absorbing structure for cycles of length $\ell$ with flexibility $m$ {which uses a template $T$} has at most $3 \ell m({\Delta(T)}+2)$ vertices, and thus, due to the condition on $\alpha$ {and the fact that we will have $\Delta(T)\le K$,} we {will} have that $|V(\cS')| \leq n/20$ throughout the proof.

  Let $T\subseteq [3m]\times [4m]$ be a bipartite template with
  flexibility~$m$ and flexible set $J_1=[2m]$ such that $\Delta(T)\le K$, as provided by Lemma \ref{lem:bipartite template}.
  Pick an arbitrary collection of~$3m$ vertex-disjoint copies of
  $\Clm$ in $V_1$ (using Fact~\ref{fact:cyclecount}). For the $i$th copy
  of $\Clm$, we label the corresponding path on $\ell-2$ edges by $P^{i}$
   (so that the ends of $P^{i}$ have $K$ common neighbours), 
  and we set $\cP_1:=\{P^1, P^2, \dots, P^{3m}\}$.  Then we label
  $A=\{a_{ij}: (i,j)\in E(T)\}$ as the vertices in the classes
  of~$K$ vertices in the copies of $\Clm$ such that each $a_{ij}$ is connected to the ends 
  of $P^i$, i.e.\ forms a copy of $C_\ell$ (we may then discard some
  extra vertices, according to the degree of~$i$ in~$T$).
   
  We will pick $Z=\{z_1,\dots, z_{4m}\}$ and
  $\cP_2=\{P_{ij}: (i,j)\in E(T)\}$ satisfying the definition of the
  absorbing structure as follows. 
  We  choose $Z$ in two phases, where all but at most $\eps p^2n$ vertices for $Z$ will be chosen in\
  the first phase.  
  We first use vertices in $V_1$. 
  We recursively do the following. 
  We pick the smallest index $j\in [4m]$
  (as long as there exists such an index)
  so that $|N_G(a_{ij}, V_1)\setminus V(\cS')|\ge \delta pn/10$ for all $i$ 
  such that $(i,j)\in T$ (there are at most $K$ such $i$). 
  We pick as~$z_{j}$ an arbitrary vertex in
  $V_2\setminus (V(\cS')\cup B_j)$, where
  $B_j$ is the set of vertices $z$ in $G$ such that
  $|(N_G(a_{ij}, V_1)\setminus V(\cS'))\cap
  N_G(z)| < \delta p^2n/20$ for some $i$ with $(i,j)\in E(T)$.  Since
  $|N_G(a_{ij}, V_1)\setminus V(\cS')|\ge \delta pn/10$ and
  $\Delta(T)\le K$, Fact~\ref{fact:irregular}~\ref{fact:irre1} with
  $U=N_G(a_{ij}, V_1)\setminus V(\cS')$ implies
  that $|B_{j}|\le 40K \delta^{-1}\eps^2 pn\le n/8$, and so such a choice always exists. 

  Having chosen $z_j$, our next aim is to construct vertex-disjoint paths $P_{ij}$ of length $\ell-2$, for each $(i,j)\in E(T)$, so    
  that the endpoints of $P_{ij}$ are adjacent to both $a_{ij}$ and $z_j$. 
  For this purpose, we would like to pick two vertices $y_1, y_2$ in 
  $U_{ij}:=(N_G(a_{ij}, V_1)\setminus V(\cS'))\cap N_G(z_j)$, which are supposed to be the ends of the path $P_{ij}$ which we are going to construct. 
  Since $z_j\notin B_j$, we have  $|U_{ij}|\ge \delta p^2n/20$.
  Letting $V_1':=V_1\setminus(V(\cS')\cup U_{ij})$, we have that $|V_1'|\ge n/8$.   
  From 
   Remark \ref{rem:easy bijumbled}, we get that $e(V_1',U_{ij})\ge p|U_{ij}||V_1'|/2$. 
  We consider two cases.   If $\ell=4$ then, since there is a vertex $w$ from $V_1'$ of degree at least 
  $p|U_{ij}|/2\ge \delta p^3n/40\ge 2$ into $U_{ij}$ (by Proposition \ref{prop:lower_bd}),  there is a path $P_{ij}$ of 
  length $2$ with ends (labeled as) $y_1$ and $y_2$ in $U_{ij}$. 
  If $\ell\ge 5$, then by Fact~\ref{fact:irregular}~\ref{fact:irre1} and the choice of $\eps$, we can find two vertices $y_1$ and $y_2\in U_{ij}$, whose degrees into $V_1'$ are at least 
  $pn/30$.  Proposition~\ref{prop:find_paths} then yields the existence of a path 
  of length $\ell-4$ with ends in $N(y_1)\cap V_1'$ and $N(y_2)\cap V_1'$. Together with $y_1$ and $y_2$ 
  this provides us with the desired path $P_{ij}$.
  
  It remains still to deal with the situation (second phase), {when there are no remaining appropriate indices $j\in[4m]$.}
  Let $\tilde{J}\subseteq [4m]$ be the set of those indices $j$ such that 
  for some $\{i, j\}\in T$ we have 
  $|N_G(a_{ij}, V_1)\setminus V(\cS')|< \delta pn/10$.
  Since $|V_1\setminus V(\cS')|\ge n/5$ we have with Fact~\ref{fact:irregular}~\ref{fact:irre1} and $\Delta(T)\le K$ 
  that $|\tilde{J}|\le K(20\eps^2 p^2 n)\le \eps p^2 n$. 
  To finish the embedding, we will use vertices in $V_3$ as well. At any point we will have that $|V(\cS') \cap V_3|\le K|\tilde{J}|\ell \le \delta pn/40$. From~\eqref{eq:rand_mindeg} we get $\deg(v,V_3\setminus V(\cS')) \ge \delta pn/10$ for all vertices $v\in V(G)$ throughout the process and we can proceed as in the two paragraphs above, using $V_3$ in place of $V_1$.

Now we {analyse} the running time. 
Firstly, we pick the copies of $\Clm$ by Fact~\ref{fact:cyclecount}. 
Secondly, to find a desired $j\in [4m]$, we check $|N_G(a_{ij}, V_1)\setminus V(\cS')|$ for all vertices $a_{ij}$; 
with such a $j$, to choose $z_j$, we search through the vertices $z$ not in $V(\cS')$ and check $|(N_G(a_{ij}, V_1)\setminus V(\cS'))\cap
  N_G(z)|$ for at most $K$ such $i$'s.
By Fact~\ref{fact:irregular}~\ref{fact:irre2}, this takes polynomial time. 
At last, we pick the desired path $P_{i j}$ of length $\ell-2$.
If $\ell=4$, then we find the vertex $w\in V_1'$ with degree $2$ to $U_{i j}$ and hence $P_{i j}$ in polynomial time. 
If $\ell\ge 5$, we find $y_1$ and $y_2$ 
and  then apply Proposition~\ref{prop:find_paths}, which runs in polynomial time. 
Therefore, the overall running time is polynomial since 
partitioning as is done by Corollary~\ref{cor:partition} works in polynomial time as well.
\end{proof}

Now we are ready to prove Theorem~\ref{thm:const_cycles}.

\begin{proof}[Proof of Theorem~\ref{thm:const_cycles}]
{Let $K=68042$.}
Let $L_0:=\min\{2^k:k\in \mathbb{N}, 2^k>L\}$ and fix  $\eps \le \eps_0:=\min\{\delta/(8000KL_0^32^{(L_0+6)}),\eps_{\ref{lem:absorbing_structure}}, \eps_{\ref{lem:connections}}\}$, where $\eps_{\ref{lem:absorbing_structure}}$ is as asserted by Lemma~\ref{lem:absorbing_structure} on input $\delta':=\delta/2L_0$, $\alpha(L)$  and $\eps_{\ref{lem:connections}}$ is as asserted by Lemma~\ref{lem:connections} on input $\beta=1/(60L(K+2))$, $\delta'$ and $L$.  
Let $n_0$ be large enough.
First, using Corollary~\ref{cor:partition}, we find a partition of the vertex set of $G$ into sets $V_1\dcup V_2$ such that $|V_1|=n/L_0$  and every vertex $v\in V(G)$ satisfies 
\begin{equation}\label{eq:mindeg}
\deg(v,V_i)\ge \delta p |V_i|/2,
\end{equation} 
for $i\in[2]$.  Here $V_2$ is taken to be the union of all other sets in the equipartition given by Corollary \ref{cor:partition},
thus $|V_2|=(L_0-1)n/L_0$.
Let $F$ be a collection of cycles of lengths in the interval $[4,L]$, whose lengths sum 
up\footnote{We can assume that $F$ has $n$ vertices as if not, we can take a supergraph by 
adding 4-cycles repeatedly. We can then remove up to three vertices from $G$ without 
affecting the properties of $G$ as in the statement of Theorem~\ref{thm:const_cycles}.} to 
$n$. There is (at least) one length $\ell\in [4,L]$ such that $F$ contains at least $n/((L-3)\ell)$ cycles $C_\ell$. We write $F=F'\dcup F_\ell$, where $F_\ell$ consists of cycles of 
length $\ell$ from $F$, while $F'$ contains all other cycles. 
We will embed $F$ into $G$ in two stages. First,  we greedily embed $F'$ into $G[V_2]$. 
This is possible since \[|V(F')|\le \frac{(L-4)n}{L-3}\le \frac{(L_0-4)n}{L_0-3} =\frac{(L_0-1)n}{L_0}-\frac{3n}{L_0(L_0-3)}=|V_2|-\frac{3n}{L_0(L_0-3)}\] and since any set of at least $3n/(L_0(L_0-3))$ vertices in $G$ contains a cycle of any length from the interval $[4,L]$ (see Fact~\ref{fact:cyclecount}).

In the second stage we are left with a vertex set $U\supseteq V_1$ such that $|V(F_\ell)|=|U|$ and $\delta(G[U])\ge \delta pn/(2L_0) \geq \delta' p|U|$,  due to~\eqref{eq:mindeg}. 
All that remains to do is to find a $C_\ell$-factor in $G[U]$. 
We are thus in a position to apply Lemma~\ref{lem:absorbing_structure} to $G[U]$, where one can check that the conditions there are satisfied with respect to $|U|$ and $\delta'$. 
Thus, in polynomial time we can construct an absorbing structure $\cS=(T, \cP_1, A, \cP_2, Z, Z_1)$ for cycles of length $\ell$ with flexibility~$m= \alpha |U|$, where $\alpha:= 1/(60L(K+2))\le \alpha(\ell)$, and a vertex set $W\subset V(G)\setminus V(\cS)$, with $|W|= |U|/4$, such that for any vertex~$v$ in~$G$, we have
  $\deg(v, W) \ge \delta' p|U|/8$. 
  Let $U_0\subset (U \setminus V(\cS))$ be the set of vertices $u$ such that $\deg(u,Z_1)\leq  {p|Z_1|}/{2}$. By 
Fact \ref{fact:irregular}~\ref{fact:irre1}, we have that $|U_0|\leq 4\eps^2 p^2 |U|^2/|Z_1|= 2 \eps^2\alpha^{-1}p^2|U|$. We first incorporate the vertices of $U_0$ into cycles of length $\ell$ using vertices of $W\setminus U_0$ by applying Lemma~\ref{lem:connections} (in polynomial time) to the pairs $\{(u,u):u\in U_0\}$. Let $\cC_1$ be the set of disjoint cycles produced by this process.
  
Now we greedily apply Fact~\ref{fact:cyclecount} to find vertex-disjoint cycles $C_\ell$ in $G[U\setminus (V(\cS)\cup V(\cC_1))]$, until we are left with a set $U_1$ of cardinality at most $2^\ell\eps n$.
What remains is to find a $C_\ell$-factor in $G[U_1\dcup V(\cS)]$. 
Recall that $\deg(u,Z_1)\ge p|Z_1|/2$ for every $u\in U_1$. The assumptions of Lemma~\ref{lem:connections} are met (in particular $|Z_1|\gg |U_1|$), and therefore, applying it to the pairs of vertices $\{(u,u):u \in U_1\}$ (to find paths through $Z_1$) we find a family $\cC_2$ of $|U_1|$ vertex-disjoint cycles $C_\ell$ that cover all of $U_1$ (and some subset of $Z_1$). Next, we greedily find, applying Fact~\ref{fact:cyclecount}, $(m-|U_1|(\ell-1))/\ell$ cycles $C_\ell$ in $Z_1\setminus V(\cC_2)$, so a set $Z'_1$ of  exactly\footnote{Note that this is possible due to divisibility conditions. Indeed it is clear that $\ell \mid (|U_1|+|V(\cS)|)$ as we look to find the remaining $C_\ell$-factor on $U_1\cup V(\cS)$. Also, Fact \ref{fact:absorbing} guarantees that $\ell \mid (|V(\cS)|-m)$ and so we can conclude that $\ell \mid (m+|U_1|)$ and hence $\ell\mid (m-|U_1|(\ell-1))$ as required.} $m$ vertices of $Z_1$ remains uncovered. But then, letting $Z_1''=Z_1\setminus Z_1'$, Fact~\ref{fact:absorbing} guarantees the existence of a $C_\ell$-factor on $V(\cS)\setminus Z_1''$. This then gives us a copy of $F$ in $G$.

Note that we applied Fact~\ref{fact:cyclecount} linearly many times. 
Moreover, we applied Corollary~\ref{cor:partition}, Lemma~\ref{lem:connections}, Fact~\ref{fact:absorbing} and Lemma~\ref{lem:absorbing_structure} constantly many times.
So we conclude that we can indeed find a copy of $F$ in polynomial time.
%
\end{proof}

\section{Proof of Theorem~\ref{thm:long_cycles}} \label{Sec: Proof of Theorem long cycles}

Before proving Theorem~\ref{thm:long_cycles}, let us sketch some of the ideas that arise in the proof. 
Firstly we will apply Lemma~\ref{lem:absorbing_structure} to show the existence of an absorbing structure $\cS=(T, \cP_1, A, \cP_2, Z, Z_1)$ for cycles of length $4$  with
  flexibility $m=\lfloor\gamma n\rfloor$, with $\gamma \leq \alpha(4)=1/(240(K+2))$, as defined in Lemma \ref{lem:absorbing_structure}.
Recall that Fact~\ref{fact:absorbing} guarantees that no matter which $m$ vertices of $Z_1$ we remove, on the rest of the vertices of $\cS$ we can find a $C_4$-factor ({which will contain exactly} $3m+|E(T)|$ copies of {$C_4$}). 
{L}et us relabel the $r:=3m+|E(T)|$ paths of length two in $\cP_1\cup \cP_2$ as $\cQ=\{Q_1, Q_2, \dots, Q_{r}\}$, let $Q_h=a_h b_h c_h$ for each $h\in [r]$ {and let $Y=Z\dcup  A$}. 
{Now the property of the absorbing structure can be rephrased as follows. After removing exactly $m$ vertices, $Z'$, from $Z_1\subset Y$, there is a perfect matching between $\cQ$ and $Y\setminus Z'$ such that if $Q_h\in \cQ$ is matched with $y\in Y$, then $a_hyc_hb_h$ forms a copy of $C_4$. In what follows, the idea is to omit an edge (for example, $a_hb_h$) from each of these $C_4$ to get paths of length three which we will connect to longer paths. The key point is that we can do this by only omitting edges in the length two paths from $\cQ$. Thus we can simply connect vertices from paths in $\cQ$ through short connecting paths. Eventually, this will lead to a longer path that will contribute to our factor and although we do not know exactly what these paths will be (as it depends on the choice of matching to $y\in Y$), the lengths of the paths and the vertices not in $Y$ are fixed.}
{More precisely, w}e will group the paths in $\cQ$ according to the desired lengths of the cycle and connect the ones in the same group e.g. connect $a_h$ with $b_{h-1}$ and connect $b_h$ with $a_{h+1}$.
At the end of the proof, by Fact~\ref{fact:absorbing} we can match every remaining vertex $y\in Y$ 
to one of the $Q_h$'s, such that ${a_h} {y} {c_h} {b_h}$ forms a copy of $P_3$ which will contribute to some longer path which in turn is part of a cycle in $F$.

\begin{proof}[Proof of Theorem~\ref{thm:long_cycles}]
{Let $K=68042$.}
Let $L \ge 8000K$ and fix $\gamma:= 1/(600(K+2))\le\alpha(4)$ with $\alpha(4)$ defined in Lemma \ref{lem:absorbing_structure}.  Next, choose $\eps \le \eps_1:=\min\{\delta/(1600000K), \eps_{\ref{lem:absorbing_structure}}, \eps_{\ref{lem:connections}}\}$, where $\eps_{\ref{lem:absorbing_structure}}$ is as asserted by Lemma~\ref{lem:absorbing_structure} on input $\delta$, $\alpha=\gamma$, $\ell=4$  and $\eps_{\ref{lem:connections}}$ is as asserted by Lemma~\ref{lem:connections} on input $\beta=\gamma$, $\delta':=\delta/16$ and $\ell=3$.  
Let $n_0$ be large enough.
Let $F$ be a graph whose components are cycles of length greater than $L$. We can assume that $v(F)\geq n-L$, otherwise we can instead consider a supergraph by adding cycles of length $L+1$.  Let $F$ consist of $t$ cycles of lengths $l_1\ge  \cdots \ge l_t$, and let $l=\sum_i^t l_i$. Note that $t\leq n/L$ and $n-L \leq l \leq n$. We will exhibit an algorithm  which finds  $F\subset G$.

Let $m=\gamma n$. Apply Lemma~\ref{lem:absorbing_structure} to get an absorbing structure $\cS=(T, \cP_1, A, \cP_2, Z, Z_1)$ for cycles of length $4$  with
  flexibility $m$ and a vertex set $W\subset V(G)\setminus V(\cS)$, with $|W|= n/4$, such that for any vertex~$v$ in~$G$, we have
  $\deg(v, W) \ge \delta p|W|/8$.  Label the vertices and paths of $\cS$ as in the discussion above. 
  In particular, recall that $r:=3m+|E(T)|$.
Let $m':=\eps n$, and let $Z'\subset Z_1$ be an arbitrary subset of size $m+2m'+4t$. Let $V_0\subset (V(G) \setminus V(\cS))\cup (Z_1\setminus Z')$ be the set of vertices $v$ such that $\deg(v,Z')\leq  \frac{p|Z'|}{2}$. 
Write $V_0:=\{v_1,v_2,\ldots,v_{|V_0|}\}$.
By Fact~\ref{fact:irregular}~\ref{fact:irre1}, we have that $|V_0|\leq 4 \eps^2\gamma^{-1}p^2n$. 
We find nonnegative integers $q_{ij}$, $i\in [t], j\in [3]$ such that the following holds:
\begin{itemize}
\item ${6}q_{i1} + 3q_{i2}+3q_{i 3}\le l_i - {10}$, for each $i\in [t]$,
\item $\sum_{i=1}^t q_{i 1} = r$, $\sum_{i=1}^t q_{i 2} = |V_0|$, and $\sum_{i=1}^t q_{i 3}= m'$.
\end{itemize}
Such a choice can be achieved easily since $r=3m+|E(T)|$ and $6r+3|V_0|+3m'\ll l-15t$. 

We now run the first phase of our algorithm:

\begin{enumerate}[label=($\roman*$)]

\item We arbitrarily partition the set $\{{\{a_h, b_h, c_h\}, {h\in [r]}}\}$ into $t$ subsets of sizes $q_{11}, q_{21}, \dots, q_{t1}$ and partition $V_0$ into $t$ subsets of sizes $q_{12},\dots, q_{t2}$. 

\item For $i\in [t]$, we fix an arbitrary linear order of the $q_{i1}$ triples of vertices and $q_{i2}$ vertices of $V_0$, and insert two new vertices $x_{i}^1, x_{i}^2$ not in $W\cup V_0\cup V(\cS)$ to the ordering, one to the beginning, one to the end.
Apply Lemma~\ref{lem:connections} to the pairs $\{{b_{h-1}, a_h}\}$ of consecutive elements from each group simultaneously (we view each single vertex $v$ in the ordering as $v={a_h=b_h}$), and get disjoint length three paths through $W\setminus V_0$ joining the pairs. 
This is possible because the number of pairs we connect is at most $2t+r+|V_0|\le 2n/L + 3m(1+K) + 4 \eps^2\gamma^{-1}p^2n\le 2n/L+3(K+2)\gamma n < n/120$, and every vertex has degree at least $\delta p|W|/8 - |V_0|\ge \delta p|W|/9$ to $W\setminus V_0$, and $|W\setminus V_0|\ge n/5$.
\end{enumerate}

For each $i\in [t]$, we obtain a sequence of paths on (in total) $5q_{i1} + 3q_{i2} + {3}$ vertices (they will become a single path of length $6q_{i1} + 3q_{i2} + {3}$ after absorbing \emph{exactly} $q_{i1}$ vertices from $Z_1$).
Next we will greedily find paths for each $i\in[t]$ which will comprise the majority of the remainder of the cycles. 
\begin{enumerate}[label=($\roman*$)]
\setcounter{enumi}{2}
\item {Fix $U$ to be the vertices in $(V(G)\setminus(V(\cS)))\cup (Z_1\setminus Z')$ which were not used in the paths chosen in the first phase}.  For $i\in [t]$, we {repeatedly} find  a path of length exactly $l_i-{6}q_{i1}-3q_{i2}-3q_{i3}-{9}$ in the uncovered vertices {of $U$} using Lemma~\ref{lem:long_paths} (for this observe that there are  at least $\ge \eps n$  unused vertices from $U$  by the choice of the parameters). 
Denote the endpoints of the path by $x_{i}^3$ and $x_{i}^4$.
 
\item Arbitrarily choose $m'$ vertices from $U$ (it could happen that there are more 
vertices in $U$ but only if $F$ has less than $n$ vertices), partition and label them in 
such a way that for each~$i$ there are $q_{i3}$ vertices $u_{i,1},\ldots,u_{i,q_{i3}}$. 

\item Apply Lemma~\ref{lem:connections} to find paths of length 3 to connect the following set of pairs 
\[
\bigcup_{i=1}^t\{(x_{i}^2, x_{i}^3),(x_{i}^4, u_{i,1}), (u_{i,1}, u_{i,2}),\dots,(u_{i,q_{i3}}, x_{i}^1)\}
\]
with inner vertices from $Z'$.
Note that this is possible as all the vertices of the pairs have good degree to $Z'$ and the number of pairs to connect is $2t+\sum_i q_{i3} = 2t + m'$, which is much less than $m=\gamma n$. 

\item In the previous step we used exactly $2m'+4t$ vertices of $Z'$ in length 3 paths. 
Thus the set $Z''\subseteq Z_1$ of unused vertices has size exactly 
$m$. 
By Fact~\ref{fact:absorbing} we can find a $C_4$-factor on $V(\cS)\setminus (Z_1\setminus Z'')$ in polynomial time. 
Note that the paths $a_j {y_j} c_j b_j$ for each $C_4$ on $\{{y_j}, a_j, b_j, c_j\}$ will complete the cycles of length exactly
\[
(6q_{i1} + 3q_{i2} + {3}) + (l_i-{6}q_{i1}-3q_{i2}-3q_{i3}-{9}) + 3q_{i3} +6 = l_i
\] 
for each $i\in [t]$.
Thus,  we have found a copy of $F$ in $G$. 
\end{enumerate}

Note that we can compute the values of $q_{ij}$ greedily in time $O(n)$. 
Each of Lemma~\ref{lem:connections}, Fact~\ref{fact:absorbing} and Lemma~\ref{lem:absorbing_structure} runs in polynomial time and we use them at most twice.
Finally, we applied Lemma~\ref{lem:long_paths} $t$ times and so  
 the overall running time is polynomial.
\end{proof}

Let us mention here that one could also define an absorbing structure specifically for the longer cycles we build in Theorem \ref{thm:long_cycles}, connecting edges into paths according to the adjacencies of a template.  Although this alternative structure would be easier to describe and would remove some of the technicalities in the above proof, we chose to instead work from the absorbing structure used for finding factors which involve short cycles, for the sake of brevity.

\section{A proof of Theorem~\ref{thm:Nenadov}} \label{Sec: Proof of Theorem Nenadov}

Nenadov's proof is algorithmic, except the proof of~\cite[Lemma 3.5]{Nen18}, in which he used a Hall-type result for hypergraphs due to Haxell. 
Here we give an alternative proof of this lemma, which moreover provides a polynomial time algorithm. 

We first need to recall some definitions from~\cite{Nen18}. 
Let $K_4^-$ be the unique graph with 4 vertices and 5 edges.
Define an $\ell$-chain as a graph obtained by sequentially identifying $\ell$ copies of $K_4^-$ on vertices of degree $2$.
Note that an $\ell$-chain contains exactly $\ell+1$ vertices such that the removal of any one of them results in a graph that has a triangle-factor.
These vertices are called \emph{removable}.

We say that a triangle in $G$ \emph{traverses} three chains $D_1$, $D_2$ and $D_3$ if it intersects all of them at some removable vertices.
Observe that if $D_1$, $D_2$ and $D_3$ are disjoint chains in $G$ and there exists a triangle in $G$ traversing them, then $G[V(D_1)\cup V(D_2)\cup V(D_3)]$ contains a triangle-factor.


Here we state~\cite[Lemma 3.5]{Nen18} and give an alternative (algorithmic) proof.

\begin{lemma}[Lemma 3.5 in~\cite{Nen18}]
Let $G$ be a $(p, \lambda)$-bijumbled graph on $n$ vertices with $\lambda\le \eps p^2 n$ for some $\eps\in (0, 1/16]$.
Suppose we are given disjoint $\ell$-chains $D_1', \dots, D_t'\subseteq G$ for some $t, \ell \in \NN$ such that $\ell$ is even, $t\ge 2000$ and $400\lambda/p^2\le t(\ell+1)\le n/24$.
Then for any subset $W\subseteq V(G)\setminus \bigcup_{i\in [t]} V(D_i')$ of size $|W|\ge n/4$ there exist disjoint $(\ell/2)$-chains $D_1,\dots, D_{2t}\subseteq G[W]$ with the following property: for every $L\subseteq [2t]$ there exists $L'\subseteq [t]$ such that
\[
G\left[ \bigcup_{i\in L}V(D_i)\cup \bigcup_{i\in L'}V(D_i') \right]
\]
contains a triangle-factor, which can be found in polynomial time.
\end{lemma}

\begin{proof} We set $\eps:=1/16$. 
Note that a similar calculation as in the proof of Fact~\ref{fact:irregular}~\ref{fact:irre1} shows that the number of vertices which have at most $\eps p t \ell$ neighbours in a set of size at least $t(\ell+1)/8\ge 50\lambda/p^2$ is at most
\begin{equation}\label{eq:lambda}
\frac{\lambda^2 t(\ell+1)/8}{(1/8-\eps)^2 p^2 t^2 (\ell+1)^2} \le \frac{\lambda}{3200(1/8-\eps)^2} < \lambda/2.
\end{equation}

Given $\ell$-chains $D_1', \dots, D_t'$, we partition them arbitrarily into four groups of almost equal sizes, $\cD_1, \dots, \cD_4$.
Note that for $\cD_3$ and $\cD_4$, since each of them contains at least $t(\ell+1)/4$ removable vertices, by~\eqref{eq:lambda} the number of vertices of $G$ that have degree less than $\eps p t \ell\le \eps p n/24$ to either of their removable vertices is at most $\lambda$.
Now we greedily pick $2t$ $(\ell/2)$-chains $D_1,\dots, D_{2t}$ in $W$ but avoiding these bad vertices by~\cite[Lemma 3.2]{Nen18}.
It remains to verify the `absorption' property.
Fix any subset $L\subset [2t]$ of $(\ell/2)$-chains $D_i$, $i\in L$.
We first greedily find triangles traversing $(\ell/2)$-chains (and thus obtain triangle-factors on them) until $t/8$ of them are left.   
Indeed, this is possible since as long as there are more than $t/8$ of them left, we can greedily partition them into three groups of size roughly $t/24$.
Because $(t/24) (\ell/2+1) > t(\ell+1)/48> 2\lambda/p^2$, it follows from~\eqref{eq:EML2} (for a proof see for example~\cite[Lemma 2.4]{Nen18}), we find a triangle with one vertex from each group.
This triangle traverses the three chains containing it and thus there is a triangle-factor covering these three chains.
So we can reduce the number of chains by $3$.

We will match the remaining $t/8$ $(\ell/2)$-chains with the $\ell$-chains.
We start with using $\ell$-chains in $\cD_1$, $\cD_2$ and recursively find triangles traversing one $(\ell/2)$-chain and one $\ell$-chain in $\cD_1$, one $\ell$-chain in $\cD_2$.
That is, as long as there exists a vertex $v$ in one of the `unmatched' chains that sends more than $\eps p t \ell$ edges to the unused removable vertices in both $\cD_1$ and $\cD_2$, then we pick an edge (whose existence is asserted by~\eqref{eq:EML2}) from these neighbourhoods, namely, a triangle containing $v$.
Note that when we stop, the vertices remaining unmatched have degree at most $\eps p t \ell$ to the unused removable vertices of either in $\cD_1$ or $\cD_2$.
Note that there are still roughly half of the chains in $\cD_1$ and $\cD_2$ left, which contain at least $(\ell+1) \cdot t/8$ removable vertices in both $\cD_1$ and $\cD_2$.
Thus, by~\eqref{eq:lambda} there are at most $\lambda$ vertices that send low degree to either of them, namely, at most $2\lambda/\ell$ $(\ell/2)$-chains are left unmatched.
Now we can proceed to match the chains greedily by $\cD_3$ and $\cD_4$.
This is possible because each time we match a chain, we consume $\ell+1$ removable vertices from $\cD_3$ and $\cD_4$, respectively, and so in total this will consume at most $(\ell+1)(2\lambda/\ell) = 2\lambda(1+1/\ell)$ removable vertices, which is much less than $\eps p t \ell$.

For the running time, note that we used~\cite[Lemma 3.2]{Nen18} in the proof, but the desired chains can be constructed by depth-first search, which can be done in polynomial time. 
We also used~\cite[Lemma 2.4]{Nen18} to claim the existence of a triangle, but we can find this triangle by brute-force searching the neighbourhood of a vertex. 
Finally, it takes polynomial time  
to decide which $v$ to use and 
to find the triangle containing $v$.
\end{proof}

\section{Concluding remarks}
In this paper we answered the question of Nenadov~\cite{Nen18} by providing a deterministic polynomial time algorithm, which finds any given $2$-factor in a $(p,\eps p^2n/\log n)$-bijumbled graph on $n$ vertices of minimum degree $\delta pn$ (for any fixed $\delta>0$), with $p>0$ and some absolute parameter $\eps=\eps(\delta)>0$. This is optimal up to the $O(\log n)$-factor. 
It also follows from the proof that the strongest condition hinges on the fact that a triangle might be present in a $2$-factor (see Theorem~\ref{thm:Nenadov}). Indeed, it follows from the proof of Theorem~\ref{thm:main} that the condition $\lambda\le \eps p^2n$ would suffice for a $2$-factor of girth at least $4$ and a solution to Conjecture \ref{conj:KSS} would imply that this condition would guarantee the existence of any $2$-factor. The celebrated construction, due to Alon~\cite{Alon94}, of triangle-free pseudorandom graphs has been extended by Alon and Kahale~\cite{AK98} to graphs without odd cycles of length $2\ell+1$. They constructed $\left(n,\Theta(n^{2/(2\ell+1)}),\Theta(n^{1/(2\ell+1)})\right)$-graphs of odd girth at least $2\ell+3$. It is proved in~\cite[Proposition~4.12]{KS06} that an $(n,d,\lambda)$-graph with $\lambda^{2\ell-1}\ll d^{2\ell}/n$ contains 
a copy of $C_{2\ell+1}$. Since $\lambda=\Omega(\sqrt{d})$ for, say $d\le n/2$, we have the lower bound on $d=\Omega(n^{2/(2\ell+1)})$. As for even cycles, a theorem of Bondy and Simonovits~\cite{BS74}, which doesn't require any bound on  $\lambda$, states that $d\gg n^{1/\ell}$ already implies the existence of $C_{2\ell}$. 
It is thus a natural avenue to further investigate the (almost) optimal conditions of when a $(p,\lambda)$-bijumbled graph contains a given $2$-factor of girth at least $\ell$.  When $\ell=n$, the best condition for $(n,d,\lambda)$-graphs is provided by the result of Krivelevich and Sudakov~\cite{KriSudHam} which gives $\lambda\le d(\log\log 
n)^2/ (1000\log n \log\log\log n)$, while another conjecture of these authors~\cite{KriSudHam} states that $\lambda\le cd$ should already be sufficient for some absolute $c>0$. This conjecture would follow from the famous toughness conjecture of Chv{\'a}tal~\cite{Ch73}, as shown by Alon~\cite{Al95}.

\bibliography{literature}
\end{document}